\documentclass[reqno]{amsart}
\usepackage{graphicx}
\usepackage{amssymb,amsmath}
\usepackage{amsthm}
\usepackage{amssymb,amsmath}
\usepackage{amsthm}
\usepackage{amsfonts}
\usepackage[mathscr,mathcal]{eucal}

\newtheorem{thm}{Theorem}[section]

\newtheorem{lem}[thm]{Lemma}

\theoremstyle{definition}

\theoremstyle{remark}

\numberwithin{equation}{section}

\newcommand{\ce}{\mathcal{C}}
\newcommand{\N}{\mathcal{N}}
\newcommand {\la}{\langle}
\newcommand {\ra}{\rangle}

\begin{document}

\title[ maximal subset of pairwise
 non-commuting elements ]{maximal subset of pairwise
 non-commuting elements of finite minimal non-abelian groups}%
\author[ S. Fouladi,  R. Orfi and A. Azad]{  S. Fouladi ,  R. Orfi and A. Azad }%

 \address{ Department of Mathematics,  Faculty of Science, Arak University, Arak 38156-8-8349, Iran.}
 
\email{ a-azad@araku.ac.ir} 
  
 \address{Faculty of Mathematical Sciences and Computer, Kharazmi University,
 50 Taleghani Ave., Tehran 1561836314, Iran.}
 \email{ s\_ fouladi@khu.ac.ir}
 
  \address{Faculty of Mathematical Sciences and Computer, Kharazmi University,
 50 Taleghani Ave., Tehran 1561836314, Iran.}
 \email{ orfi@khu.ac.ir}

\date{\today}

\begin{abstract} Let $G$ be a group. A subset $X$ of $G$ is a set of pairwise non-commuting elements
if $xy\not=yx$ for any two distinct elements $x$ and $y$ in $X$.
If $|X|\geq |Y|$ for any other set of pairwise non-commuting
elements $Y$ in $G$, then $X$ is said to be a maximal subset of
pairwise non-commuting elements. In this paper we determine the
cardinality of a maximal subset of pairwise non-commuting elements
for finite minimal non-abelian groups.
\end{abstract}

\maketitle

\section{\Large{Introduction}}

\vspace*{0.4cm}
 Let $G$ be a non-abelian group and let  $X$  be a  maximal subset of pairwise non-commuting
elements of $G$.  The cardinality of such a
  subset is denoted by $\omega(G)$. Also $\omega(G)$ is the
  maximal clique size in the non-commuting graph of a group $G$.
  Let $Z(G)$ be the center of $G$. The non-commuting graph of a group $G$ is a graph with
  $G \backslash Z(G)$ as the vertices and join two distinct vertices $x$ and
  $y$, whenever $xy\neq yx$. By a famous result of B. H. Neumann \cite{Neumann},
answering a question
 of P. Erd$\ddot{o}$s,  the finiteness of $\omega (G)$ in
 $G$ is equivalent to the finiteness of the factor group
 $G/Z(G)$.
 Pyber \cite{Pyber} has shown that there is some constant $c$ such that $|G : Z(G)|\leq c^{\omega
 (G)}$.\break
Moreover various attempts have been made to find $\omega(G)$ for
some groups $G$, see for example \cite{Abdollahi}, \cite{Azad},
\cite{AFO}, \cite{Chin}, \cite{Fouladi} and \cite{Fo}.\\
In this paper we find $\omega(G)$ for any finite minimal
non-abelian group. A minimal non-abelian group is a non-abelian
group such that all its proper subgroups are abelian. A useful
structure of these groups is given in [\ref{Huppert}, Aufgaben
III. 5.14], which states that the order of a minimal non-abelian
group $G$ has at most two distinct prime divisors and if $G$ is
not a $p$-group, then only one of its sylow subgroup is normal in
$G$. Following \cite[Lemma 116.1 (a)]{Berk}, we see that
$\omega(G)=p+1$ for any finite minimal non-abelian $p$-group $G$.
Therefore in this paper we assume that $G$ is a finite minimal
non-abelian group, which is not a $p$-group and we show that
$\omega(G)=|Q|+1$, where $Q$ is  the normal $q$-Sylow subgroup of
$G$.\\
Throughout this paper we use the following notation. $p$ denotes a
prime number. $\ce_{G}(x)$ is the centralizer of an element $x$ in
a group $G$. A group $G$ is called an $AC$-group if the
centralizer of every non-central element of $G$ is abelian.

\vspace*{0.4cm}

\section{\Large{Main result}}

\vspace*{0.4cm} First we state two following lemmas that are
needed for the main result of this paper.

\begin{lem}\label{1} The following conditions on a group $G$ are
equivalent.
\begin{itemize}
\item[(i)] $G$ is an $AC$-group.
\item[(ii)] If $[x,y]=1$ then
$\ce_G(x)=\ce_G(y)$, where $x,y\in G\backslash Z(G)$.
\end{itemize}
\end{lem}

\begin{proof}
This is straightforward. See also [\ref{R}, Lemma 3.2].
\end{proof}

\begin{lem}\label{2} \cite[Lemma 2.3]{AFO} Let $G$ be  an $AC$-group.
\begin{itemize}
\item[(i)] If $a, b \in G\backslash Z(G)$ with distinct
centralizers, then $\ce_G(a)\cap \ce_G(b)=Z(G)$.
 \item[(ii)] If $G=\cup_{i=1}^k \ce_G(a_i)$, where $\ce_G(a_i)$ and $\ce_G(a_j)$
are distinct proper subgroups of $G$ for $1\leq i<j\leq k$, then
$\{a_1\dots a_k\}$ is a maximal set of pairwise non-commuting
elements in $G$.
\end{itemize}
\end{lem}

Now we find $\omega(G)$, for a finite minimal
non-abelian group $G$ in which $G$ is not a $p$-group.\\
The following theorem gives a structure for finite minimal
non-abelian groups which play  an important role in our proof of
the main Theorem.

\begin{thm}\label{4}\textup{[\ref{Huppert}, Aufgaben III. 5.14 ]}. Let $G$ be a finite
minimal non-abelian group. Then
\begin{itemize}
\item[(i)] the order of $G$ has at most two distinct prime divisors,
\item[(ii)] if $|G|$ is not a power of a prime then $G=PQ$, where $P$ is a
cyclic $p$-Sylow subgroup of $G$ and $Q$ is the elementary abelian
minimal normal $q$-Sylow subgroup of $G$.
\end{itemize}
\end{thm}

\begin{lem}\label{7} Let $G$ be a finite minimal non-abelian group
and $G=PQ$, where $P$ is a cyclic $p$-Sylow subgroup of $G$ and
$Q$ is the elementary abelian minimal normal $q$-Sylow subgroup of
$G$. Then
\begin{itemize}
\item[(i)] $G'=Q$,
\item[(ii)] $G'\cap Z(G)=1$ and so $Z(G)$ is a $p$-subgroup of
$G$,
\item[(iii)] $\ce_{G}(P)=\N_{G}(P)=P$,
\item[(iv)] $\ce_{G}(b)=Z(G)\times Q$ for any $1\neq b \in Q$.
\end{itemize}
\end{lem}

\begin{proof}(i) $G'\leq Q$ since $G/Q\cong P$. Now the
result follows from the fact that $Q$ is minimal normal subgroup
of $G$.\\
(ii) We have $G'\cap Z(G)=Q \cap Z(G)$ is normal in $G$ and if
$Q\leq Z(G)$, then $G$ is abelian which is impossible. This yields
that  $Q \cap Z(G)=1$ and so
$Z(G)$ is a $p$-subgroup of $G$.\\
(iii) If $P\lneqq \N_{G}(P)$, then  there exists $ x \in
\N_{G}(P)$ of order $q$. Hence $x \in Q$. Let $P=\la a \ra$, then
$[a, x] \in P$ and so $[a, x]=1$ by (i). This implies that $x \in
Z(G)$. Therefore $x=1$ by
(ii) and so $\N_{G}(P)=P$. The rest is obvious.\\
(iv) If $1\neq b \in Q$, then by (ii), $b \in Q\setminus Z(G)$
 and so $\ce_{G}(b)\lneqq G$ is abelian. Since $
Q \leq \ce_{G}(b)$, we may write $\ce_{G}(b) \cong Q\times P_{0}$,
where $P_{0}$ is the $p$-Sylow subgroup of $\ce_{G}(b)$. Therefore
$[P_{0}, Q]=1$. Moreover $P_{0}\leq P^g$ for some $g \in G$ and
$G=P^{g}Q$, which implies that $P_{0}\leq Z(G)$. Furthermore
$Z(G)\leq P_{0}$ by (ii) and the fact that $Z(G)\leq \ce_{G}(b)$.
This complete the proof.
\end{proof}

\begin{thm}\label{8} Let $G$ be a finite minimal non-abelian group
and $G=PQ$, where $P$ is a cyclic $p$-Sylow subgroup of $G$ and
$Q$ is the elementary abelian minimal normal $q$-Sylow subgroup of
$G$. Then $\omega(G)=|Q|+1$.
\end{thm}

\begin{proof} Let $|G|=p^{\alpha}q^{\beta}$ and $P_{1}=P,
 P_{2} , \dots , P_{m}$ be all distinct  $p$-Sylow subgroups of $G$ and
$P_{i}=\la a_{i} \ra$ for $1\leq i \leq m$. Obviously
$m=q^{\beta}$ and $\ce_{G}(a_{i})=P_{i}$ by Lemma \ref{7}(iii)
and so  $a_{i} \notin Z(G)$ since $G=P_{i}Q$. Now let $1\neq b \in Q
$, then $ b \notin Z(G)$ by Lemma \ref{7}(ii). Moreover for $1\leq
i \leq m$  we have $\ce_{G}(a_{i})\neq\ce_{G}(b)$, for otherwise
we see that $b \in Z(G)$, which is a contradiction. Now we
calculate the order of $A=\ce_{G}(a_{1})\cup\dots
\cup\ce_{G}(a_{m})\cup\ce_{G}(b)$. By Lemma \ref{2}(i) and the
fact that $G$ is an AC-group, we see that $|A|=\sum_{i=1}^m
(|\ce_{G}(a_i)|-|Z(G)|)+|\ce_{G}(b)|$. Moreover by Lemma
\ref{7}(iii), (iv) we have $|\ce_{G}(a_{i})|=|P_{i}|=p^{\alpha}$
and $|\ce_{G}(b)|=|Z(G)| q^{\beta}$. Therefore $|A|=|G|.$ This
yields that
$G=\ce_{G}(a_{1})\cup\dots\cup\ce_{G}(a_{m})\cup\ce_{G}(b)$, and
so $\omega(G)=|Q|+1$ by Lemma \ref{2}(ii).
\end{proof}



\vspace*{1cm}


\begin{thebibliography}{99}

\bibitem{Abdollahi}A. Abdollahi, A. Akbari and H. R. Maimani, Non-commuting graph of a
group, {\it J. Algbera} {\bf 298} (2006), no. 2, 468-492.

\bibitem{Azad}A. Azad, Cheryl E. Praeger, Maximal subsets of  pairwise
 non-commuting elements of three-dimensional general linear
 groups, {\it Bull. Aus. Math. Soc.} {\bf 80} (2009), no. 1, 91-104.

\bibitem{AFO}A. Azad, S. Fouladi and R. Orfi, Maximal subsets of  pairwise
 non-commuting elements of some finite $p$-groups, {\it
 Bull. Iran. Math. Soc.} (to be appear)

\bibitem{Berkovich}\label{Berkovich} Y. Berkovich, {\it Groups of Prime Power
Order} Vol. 1, Walter de Gruyter, Berlin, 2008.

\bibitem{Berk} Y. Berkovich and Z. Janko, {\it Groups of Prime Power
Order} Vol. 3, Walter de Gruyter, Berlin, 2011.

\bibitem{Isaacs}E. A. Bertram, Some applications of graph theory to
finite groups, {\it Discrete Math.} {\bf 44} (1983), no. 1, 31-43.

\bibitem{Chin} A. M. Y. Chin, On non-commuting sets in an extraspecial $p$-group, {\it J. Group Theory} {\bf 8} (2005), no. 2, 189-194.

\bibitem{Fouladi} S. Fouladi and R. Orfi, Maximal subsets of  pairwise
 non-commuting elements of some $p$-groups of maximal class, {\it
 Bull. Aust. Math. Soc.} {\bf 84} (2011), no.3, 447-451.

\bibitem{Fo} S. Fouladi and R. Orfi, Maximum size of subsets of pairwise
 non-commuting elements in finite metacyclic $p$-groups, {\it
 Bull. Aust. Math. Soc.} (to be appear)


\bibitem{HuppertI}\label{Huppert}B. Huppert, {\it Endliche Gruppen, I} (Springer-Verlag, Berlin, 1967).

\bibitem{Ta} M. Mashkouri and B. Taeri , On a graph associated to groups, {\it  Bull. Malays. Math. Sci. Soc.} (2) {\bf 34} (2011), no. 3, 553-560.

\bibitem{Neumann}B. H. Neumann, A problem of Paul Erd$\ddot o$s on groups,
{\it J. Aust. Math. Soc. Ser. A} {\bf 21} (1976), no. 4, 467-472.

\bibitem{Pyber} L. Pyber, The number of pairwise non-commuting
elements and the index of the centre in a finite group, {\it J.
London Math. Soc.}  {\bf 35} (1987), no. 2, 287-295.

\bibitem{Rocke}\label{R} D. M. Rocke, $p$-groups with abelian centralizers,
{\it Proc. London math. Soc.}  {\bf 30} (1975), no. 3, 55-57.


\end{thebibliography}
\end{document}